\documentclass[aip,graphicx,amsmath,amssymb]{revtex4-1}
\usepackage{amsmath,amsthm,amsfonts,amssymb,layout,indentfirst}
\usepackage{graphics,epsfig}

\draft 
\newtheorem{theorem}{Theorem}[section]
\newtheorem{lemma}[theorem]{Lemma}

\theoremstyle{definition}

\newtheorem{proposition}[theorem]{Proposition}

\theoremstyle{remark}

\usepackage{array}
\usepackage{amssymb}
\usepackage{latexsym}
\usepackage{caption}
\usepackage{amsfonts}
\usepackage{tikz}

\usepackage{placeins}

\theoremstyle{plain} \numberwithin{equation}{section}
\begin{document}


\title{Some Studies On Central Derivation of Nilpotent Lie Superalgebras  } 



\author{Rudra Narayan Padhan}

\email{rudra.padhan6@gmail.com}
\affiliation{Department of Mathematics, National Institute of
Technology Rourkela, Odisha- 769008, India}
\author{K.C. Pati}
\email[Corresponding Author email: ]{kcpati@nitrkl.ac.in}
\affiliation{Department of Mathematics, National Institute of
Technology Rourkela, Odisha- 769008, India
}%
\begin{abstract}
Many theorems and formulas of Lie algebras run quite parallel to Lie superalgebra case, sometimes giving interesting results. So it is quite natural to extend the new concepts of Lie algebra immediately to Lie superalgebra case, as these type of algebras have wide applications in physics and related theories. Using the concept of isoclinism, F. Saeedi and S. Sheikh-Mohseni recently studied the central derivation of  nilpotent Lie algebra  with nilindex 2. The purpose of the present paper is to continue and extend the investigation to obtain some similar results for Lie superalgebras, as isoclinism in Lie superalgebra is being recently introduced.

\end{abstract}
\maketitle 

\section{Introduction and preliminaries}
In this section we review some terminology on Lie superalgebra and recall notations used in the paper.
A $superaglebra$[8] is a $\mathbb{Z}_{2}$-graded algebra $A = A_{\overline{0}} \oplus A_{\overline{1}}$ 
(that is, if  $a\in A_{\alpha}$, $b \in A_{\beta}$, $\alpha,\beta \in \mathbb{Z}_{2}$, then $ ab\in A_{\alpha+\beta}$) . 
A $Lie~superalgebra$ is a superalgebra $G=G_{\overline{0}}\oplus G_{\overline{1}}$ with the operation [.,.] satisfying the following axiom:
 \begin{enumerate}
 \item $[a,b]=-(-1)^{deg(a)deg(b)}[b,a]$ $~~~~~~~$for $a\in G_{\alpha}$ and $b\in G_{\beta}$
 \item $[a,[b,c]]=[[a,b],c]+(-1)^{deg(a)deg(b)}[b,[a,c]]$  $~~~~~~~$for $a\in G_{\alpha}$ and $b\in G_{\beta}$
  \end{enumerate}
 Where $deg(a)$ is $0$ if $a\in G_{\overline{0}}$ and $deg(a)$ is $1$ if $a\in G_{\overline{1}}$. For a Lie superalgebra $G=G_{\overline{0}}\oplus G_{\overline{1}}$, the even part $G_{\overline{0}}$ is a Lie algebra and $G_{\overline{1}}$ is a $G_{\overline{0}}$-module. A $\mathbb{Z}_{2}$-graded vector subspace $F$ of $G$ is called a super subalgebra (respectively ideal) of $G$, if $[F,F]\subseteq F$ (respectively $[F,G]\subseteq F$ ). The ideal $Z(G)=\{x \in G~|~[g,x]=0, ~\forall~g \in G\}$ is called the center of the Lie superalgebra $G$.  The Lie superalgebra $G$ is abelian, if $[G,G]=0$. Since $[G_{\overline{0}},G_{\overline{0}}]\subseteq G_{\overline{0}}$ and $[G_{\overline{1}},G_{\overline{1}}]\subseteq G_{\overline{0}}$,   we can observe that, a Lie superalgebra without even part, i.e.,  $G_{\overline{0}}=0$, is an abelian Lie superalgebra.\\  
 
 By a $homomorphism$ between super vector spaces $T:G \rightarrow H$ of degree $\beta \in \mathbb{Z}_{2}$, we mean a linear map satisfying $T(G_{\alpha})\subseteq G_{\alpha+\beta}$, $\forall~\alpha \in \mathbb{Z}_{2} $. In particular, if the degree of $T$ is zero, then the homomorphism $T$ is called homogeneous linear map of even degree. A Lie superalgebra homomorphism $T:G \rightarrow H$ is a homogeneous linear map of even degree such that $T[g_{1},g_{2}]=[T(g_{1}),T(g_{2})], ~\forall~ g_{1},~g_{2}~ \in G$. The notation of $epimorphisms,~isomorphisms,~automorphisms$ have the obvious meaning. For superdimension of Lie superalgebra $G$, we simply write $dim(G)=(m|n)$, where $dim(G_{\overline{0}})=m$ and $dim(G_{\overline{1}})=n$. For Lie superalgebra $G$, $End(G)$ consists of set of all linear transformation from $G$  to $G$, which has a $\mathbb{Z}_{2}$-grading by defining;
 \[ End_{\alpha}(G)=\{T \in End(G)~|~T(G_{\beta})\subseteq G_{\alpha+\beta} \}~\forall~\alpha,~\beta \in \mathbb{Z}_{2}. \]
 Also, $End(G)=End_{\overline{0}}(G)\oplus End_{\overline{1}}(G)$ is a Lie superalgebra, where the Lie super bracket is define by;
 \[ [T_{1},T_{2}]=T_{1}\circ T_{2}-(-1)^{deg(T_{1})deg(T_{2})}T_{2} \circ T_{1},~ \forall ~ T_{1},~T_{2} \in End(G).\]
 
 A $derivation$ of degree $\alpha,~\alpha\in \mathbb{Z}_{2}$, of a Lie superalgebra $G$ is an endomorphism $T \in End_{\alpha}(G)$ with the property 
 \[T[g_{1},g_{2}]=[T(g_{1}),g_{2}]+(-1)^{deg(\alpha)deg(g_{1})}[g_{1},T(g_{2})],~\forall~g_{1} \in G_{\beta},~ g_{2} \in G,~\beta\in \mathbb{Z}_{2} .\]
 We denote $SDer_{\alpha}(G)\subseteq End_{\alpha}(G)$ the space of all derivation of degree $\alpha $, and we set $SDer(G)= SDer_{\overline{0}}(G)\oplus SDer_{\overline{1}}(G)$. $SDer(G)$ is a Lie subalgebra of $End(G)$. The map $ad_{g}:G\rightarrow G$ defined by $g \mapsto [g,x]$ is a derivation called the $inner ~derivation$ of $G$.
 The space $SIDer(G)=\{ad_{g}~|~g \in G\}$ of inner derivation is an ideal of $SDer(G)$, because $[T,ad_{g}]=ad_{T(g)},~\forall~T \in SDer(G)$. A derivation of a Lie superalgebra $G$ is called a $central~ derivation$ if its image is contained in the center of $G$. The set of all central derivation is denoted by $SDer_{z}(G)$ is a subalgebra of $SDer(G)$. A Lie superalgebra $G$ is a $stem$ Lie superalgebra, if $Z(G)\subseteq G^{1}.$\\
 
 The $lower~central~series $ of a Lie superalgebra $G$ is defined as follows:
  \[G=G^{0}\supseteq G^{1}\supseteq\cdots \supseteq G^{n}\supseteq\cdots ,\]
  where $G^{1}$ is the derived algebra of $G$ and $G^{n}=[G^{n-1},G]$. A Lie superalgebra $G$ is nilpotent if $\exists$ a positive integer $n$ such that $G^{n}={0}$. The smallest integer $n$ for which $G^{n}=0$ and $G^{n-1}\neq 0$ is called the $nilindex$ of $G$.\\

 \textbf{\textit{Observation:}} If $T$ is a central derivation, then $T(G^{1})=0$ and $[T,ad_{g}](x)=ad_{T(g)}(x)=[T(g),x]=0,~\forall~ad_{g}\in SIDer(G),~x \in G$. Hence a central derivation of Lie superalgebra commute with every inner derivation.\\
 
 Let $G$  be a Lie superalgebra and $H$ be an abelian Lie superalgebra. $Hom(G,H)$ be the set of all linear transformation from $G$ to $H$. $Hom(G,H)$ has a $\mathbb{Z}_{2}$-grading by defining;
 \[ Hom(G,H)_{\alpha}=\{T \in Hom(G,H)~|~T(G_{\beta})\subseteq H_{\alpha+\beta} \}~\forall~\alpha,~\beta \in \mathbb{Z}_{2}. \]
 Also, $Hom(G,H)=Hom(G,H)_{\overline{0}}\oplus Hom(G,H)_{\overline{1}}$ is a Lie superalgebra, where the Lie super bracket is defined by the Lie super bracket of $H$ ;
 \[[T_{1},T_{2}](g)=[T_{1}(g),T_{2}(g)]_{H}=0\]
 $ \forall~T_{1} \in Hom(G,H)_{\alpha} ,~T_{2} \in Hom(G,H)_{\beta},~g \in G_{\gamma}$
 and $[,]_{H}$ denoted Lie super bracket on $H$.\\

  Two Lie superalgebra $G$ and $H$ are $isoclinic$, if there exists a pair of isomorphisms $\varphi:G/Z(G)\longrightarrow H/Z(H)$ and $\theta:G^{1}\longrightarrow H^{1}$ such that the following diagram is commutative.\\
  \begin{center}
  \begin{tikzpicture}[>=latex]
\node (x) at (0,0) {\(G/Z(G)\times G/Z(G) \)};
\node (z) at (0,-3) {\(H/Z(H)\times H/Z(H)\)};
\node (y) at (3,0) {\(G^{1}\)};
\node (w) at (3,-3) {\(H^{1}\)};
\draw[->] (x) -- (y) node[midway,above] {$\mu$};
\draw[->] (x) -- (z) node[midway,left] {$\varphi^{2}$};
\draw[->] (z) -- (w) node[midway,below] {$\rho$};
\draw[->] (y) -- (w) node[midway,right] {$\theta$};
\end{tikzpicture}\\
 \end{center} 
  
  Where $\mu$ and $\rho$ are defined by $(\overline{x},\overline{y})\mapsto [x,y]$. In this case we say that $G$ is isoclinic to $H$ and write $G\sim H$, is a equivalence relation. The pair $(\varphi,\theta)$ is called an $isoclinism$ between $G$ and $H$. It is introduced by Saudamini [7] for Lie superalgebra case. 
  
 Moneyhun [2] proved that each isoclinism class of Lie algebras, with respect to above equivalence relation, contains a $stem~algebra$. Similar results for Lie superalgebra case has proved by Saudamini [7].\\
 
 The $Frattini~subalgebra$ of a Lie algebra is the intersection of all maximal subalgebras of the Lie algebra, which was introduced by Marshall [1]. Then later on $Frattini~subalgebra$ of a Lie superalgebra is studied by Liangyun Chen and Daoji Meng [6].  The $Frattini~subalgebra$ of a Lie superalgebra $G$, denoted by $F(G)$, is the intersection of all maximal subalgebras of the Lie superalgebra.
\section{The Main Results}
\begin{proposition}
If $G$ is a stem Lie superalgebra, then $SDer_{z}(G)$ is an abelian Lie superalgebra.
\end{proposition}
\begin{proof}
Let $T_{1} \in (SDer_{z}(G))_{\alpha},~T_{2}\in (SDer_{z}(G))_{\beta}$. Then $[T_{1},T_{2}](g)=T_{1}\circ T_{2}(g)-(-1)^{deg(T_{1})deg(T_{2})}T_{2} \circ T_{1}(g)=0,~ \forall ~g \in G$ by the observation. Hence $SDer_{z}(G)$ is abelian.
\end{proof}

\begin{proposition}
Let $G$ be a non-abelian nilpotent Lie superalgebra of finite dimension. Then $SDer_{z}(G)$ is abelian if and only if $G$ is a stem Lie superalgebra.
\end{proposition}
\begin{proof}
Let $Sder_{z}(G)$ be an abelian Lie superalgebra. As $G$ is nilpotent and $0\neq G^{1}$ is an ideal of $G$, than $G^{1} \cap Z(G)\neq 0$. Suppose $\exists,~y \in Z(G)$ such that $y\notin G^{1}$ and $0\neq x\in G^{1} \cap Z(G)$. Now we can construct two central derivation $T_{1},~T_{2}\in SDer_{z}(G)$ such that $[T_{1},T_{2}]\neq0$.\\  
 
Define $T_{1}:G \rightarrow G$ and $T_{2}:G \rightarrow G$ by
\[T_{1}(z):= \begin{cases} 
      y & if~z= y \\
      0 & otherwise
   \end{cases}
\] 
and \[T_{2}(z):= \begin{cases} 
      x & if~z= y \\
      0 & otherwise 
   \end{cases}
\]  
respectively. It is easy to check $T_{1},~T_{2}\in SDer_{z}(G)$. Now $[T_{1},T_{2}](y)=T_{1} \circ T_{2}(y)-(-1)^{deg(T_{1})deg(T_{2})}T_{2} \circ T_{1}(y)=T_{1}(x)-(-1)^{deg(T_{1})deg(T_{2})}T_{2}(y) \neq 0$. Hence we get a contradiction to our assumption. Conversely, suppose $G$ is a stem algebra, then by proposition 2.1, $SDer_{z}(G)$ is abelian.
\end{proof}

\begin{lemma}
Let $(\varphi,\theta)$ be an isoclinism between the Lie superalgebras $G$ and $H$. If $G$ is a stem Lie superalgebra, then $\theta$ maps $Z(G)$ onto $ H^{1} \cap Z(H)$.
\end{lemma}
 \begin{proof}
 Let $g\in (Z(G))_{\gamma}\subseteq (G^{1})_{\gamma}$, then $g$ can be written as $g=\sum_{i=1}^{n} \alpha_{i}[x_{i},y_{i}]$, where the $\alpha_{i}$ are scalars and $x_{i} \in G_{\alpha},~y_{i} \in G_{\beta}$ such that $\alpha+\beta=\gamma$. Let $\varphi(x_{i}+Z(G))=h_{i}+Z(H)$ and $\varphi(y_{i}+Z(G))=k_{i}+Z(H)$, where $h_{i} \in H_{\alpha},~k_{i} \in H_{\beta}$. As $\theta$ is an isomorphism, so $\theta(g)\in (H^{1})_{\gamma}$. We want to show $\theta(g)\in (Z(H))_{\gamma}$. Consider,
 \begin{equation}
\begin{split}
\theta(g)+ Z(H)&=\theta(\sum_{i=1}^{n} \alpha_{i}[x_{i},y_{i}])+Z(H)\\
&=\sum_{i=1}^{n} \alpha_{i}.\theta([x_{i},y_{i}])+Z(H)\\
&=\sum_{i=1}^{n} \alpha_{i}.\theta\circ\mu(x_{i}+Z(G),y_{i}+Z(G))+Z(H)\\
&=\sum_{i=1}^{n} \alpha_{i}.\rho\circ\varphi^{2}(x_{i}+Z(G),y_{i}+Z(G))+Z(H)\\
&=\sum_{i=1}^{n} \alpha_{i}.\rho(h_{i}+Z(H),k_{i}+Z(H))+Z(H)\\
&=\sum_{i=1}^{n} \alpha_{i}.[h_{i},k_{i}]+Z(H)\\
&=\varphi(\sum_{i=1}^{n} \alpha_{i}.[x_{i},y_{i}]+Z(G))= Z(H).
\end{split}
\end{equation}
  Hence $\theta(g)\in (Z(H))_{\gamma}$, and $\theta$ is an isomorphism, so is onto\\
 
 \end{proof}
 
 \begin{proposition}
 Let $G$ and $H$ be two isoclinic Lie superalgebras and $G$ be a stem Lie superalgebra. Then every $T \in SDer_{z}(G)$ induces a central derivation $T^{*}\in SDer_{z}(H)$.  Moreover, the map $T\rightarrow T^{*}$ is a monomorphism from $SDer_{z}(G)$ into $SDer_{z}(H)$.
 \end{proposition}
 \begin{proof}
 Let $(\varphi,\theta)$ be an isoclinism between the Lie superalgebras $G$ and $H$. Given any $h \in H_{\alpha}$, $\exists$ an element $g \in G_{\alpha}$ such that $\varphi(g+Z(G))=h+Z(H)$. By lemma 2.3, $\theta(T(g)) \in (H^{1} \cap Z(H))_{\alpha+\beta}$ if $T \in (SDer_{z}(G))_{\beta}$. Define $T^{*}:H\rightarrow H$ by $T^{*}(h)=\theta(T(g))$. Since $\theta~ and~ T$ are well-defined and by remarks(?), we can see that $T^{*}$ is well-defined. Also, if $T \in (SDer_{z}(G))_{\beta}$, then $T^{*} \in (SDer_{z}(H))_{\beta}$. Defined $\pi:SDer_{z}(G)\rightarrow SDer_{z}(H)$ by $\pi(T)=T^{*}$. Suppose $T_{1}=T_{2}\Leftrightarrow T_{1}(g)= T_{2}(g)~\forall~g \in~G_{\alpha}  \Leftrightarrow \theta(T_{1}(g))= \theta(T_{2}(g))\Leftrightarrow T^{*}_{1}(h)=T^{*}_{2}(h)~\forall ~ h \in ~H_{\alpha}\Leftrightarrow \pi(T_{1})=\pi(T_{2})$. Hence $\pi$ is a well-defined, injective and homogeneous linear map of even degree. For  $T_{1} \in (SDer_{z}(G))_{\beta}$, $T_{2} \in (SDer_{z}(G))_{\gamma}$  and by the observation, we have 
 \[ \pi([T_{1},T_{2}])((h))=[T_{1},T_{2}]^{*}(h)=\theta([T_{1},T_{2}](g))=\theta(T_{1} \circ T_{2}(g)-(-1)^{deg(T_{1})deg(T_{2})}T_{2} \circ T_{1}(g))=0.\]
 On the other hand,
 \[[\pi(T_{1}),\pi(T_{2})](h)=[T^{*}_{1},T^{*}_{2}](h)=T^{*}_{1}\circ T^{*}_{2}(h)-(-1)^{deg(T^{*}_{1})deg(T{*}_{2})}T^{*}_{2} \circ T^{*}_{1}(h)=0.\]
 Since $T^{*}_{1}(h)  \in (H^{2} \cap Z(H))_{\alpha+\beta},~T^{*}_{2}(h) \in (H^{1} \cap Z(H))_{\alpha+\gamma}$. Therefore $\pi([T_{1},T_{2}])=[\pi(T_{1}),\pi(T_{2})]$.
 
 \end{proof}
 
 \begin{proposition}
  Let $G$ be a Lie superalgebra, then $SDer_{z}(G)\cong T(G/G^{1},Z(G))$ as super vector space and it is a Lie super isomorphism, if $G$ is a stem Lie superalgebra.
 \end{proposition}
   
   \begin{proof}
   Let $T\in (Sder_{z}(G))_{\alpha}$ and define $T^{\prime}:G/G^{1}\rightarrow Z(G)$ by $T^{\prime}(g+G^{1})=T(g)~\forall~g\in G$. Clearly, $T^{\prime}$ is a linear transformation and $T^{\prime} \in (T(G/G^{1},Z(G))_{\alpha}$. Now, define $\pi:SDer_{z}(G) \rightarrow T(G/G^{1},Z(G))$ by $\pi(T)=T^{\prime}$. If $T_{1}(g)=T_{2}(g)\Leftrightarrow T^{\prime}_{1}(g+G^{1})= T^{\prime}_{2}(g+G^{1})\Leftrightarrow \pi(T_{1})=\pi(T_{2}),~\forall~g \in G$. Hence, $\pi$ is well-defined, injective and graded preserving linear transformation. The surjectivity of $\pi$ can be easily proved. If $Z(G)\subseteq G^{1}$, $T_{1}\in (SDer_{z}(G))_{\alpha}$ and $T_{2}\in (SDer_{z}(G))_{\beta}$, we have 
   \[\pi([T_{1},T_{2}])(g_{\gamma}+G^{1})= [T_{1},T_{2}]^{\prime}(g_{\gamma}+G^{1})=[T_{1},T_{2}](g_{\gamma})= T_{1}\circ T_{2}(g_{\gamma})-(-1)^{deg(T_{1})deg(T_{2})}T_{2}\circ T_{1}(g_{\gamma})=0.\]
   On the other hand,
   \[[\pi(T_{1}),\pi(T_{2})](g_{\gamma}+G^{1})=[T_{1}^{\prime},T_{2}^{\prime}](g_{\gamma}+G^{1})=[T_{1}^{\prime}(g_{\gamma}+G^{1}),T_{2}^{\prime}(g_{\gamma}+G^{1})]=0.\]
   Hence, $\pi$ is an Lie super isomorphism.
   \end{proof}
   
   \begin{lemma}
   For any arbitrary Lie superalgebra $H$, the Lie superalgebra $SDer_{z}(H)$ has a central subalgebra $N$ isomorphic to $T(G/G^{1},Z(G))$ for some stem Lie superalgebra $G$ isoclinic to $H$. 
   \end{lemma}
 \begin{proof}
 By [7, Lemma 4.2] there exist a stem Lie superalgebra $G$ isoclinic to $H$. According to the proposition 2.4, we have the embedding $\pi:SDer_{z}(G)\rightarrow SDer_{z}(H)$ defined by $\pi(T)=T^{*}$. Let $h_{\alpha} \in H_{\alpha}$, then $\exists$ $g_{\alpha} \in G_{\alpha}$ such that $\varphi(g_{\alpha}+Z(G))=h_{\alpha}+Z(H)$. Let $T^{\prime} \in (SDer_{z}(H))_{\beta}$ and  $T \in (SDer_{z}(G))_{\gamma}$. As $\theta(T(g_{\alpha}))\in (H^{1})_{\alpha+\gamma}$, so $T^{\prime}(\theta(T(g_{\alpha})))=0$. Since $T^{\prime}(h_{\alpha}) \in (Z(H))_{\alpha+\beta}$ and $T^{*}(h_{\alpha})=\theta(T(g_{\alpha}))$, we have 
 \[[T^{\prime},T^{*}](h_{\alpha})=T^{\prime}(T^{*}(h_{\alpha}))-(-1)^{deg(T^{\prime})deg(T^{*})}T^{*}(T^{\prime}(h_{\alpha}))=0\]
 Hence $N=\{T^{*}|T \in SDer_{z}(G)\}\cong SDer_{z}(G)$. Also, by proposition 2.5, we have $SDer_{z}(G)\cong T(G/G^{1},Z(G))$, as required.

\end{proof}   

\begin{lemma}
 Let $G$ and $H$ be two isoclinic Lie superalgebras. If $G$ is nilpotent with nilindex n, then so is $H$.
\end{lemma}
\begin{proof}
Let $G$ be nilpotent Lie superalgebra with nilindex $n$. Then $G^{n}=[G^{n-1},G]=0$, i.e for all $g_{1} \in G_{\alpha},~g_{2} \in G_{\beta}$, we have $(ad_{g_{1}})^{n}(g_{2})=0$. We will show, $(ad_{h_{1}})^{n}(h_{2})=0$ for all $h_{1} \in H_{\alpha},~h_{2} \in H_{\beta}$. For  $h_{1} \in H_{\alpha},~h_{2} \in H_{\beta}$, $\exists~g_{1} \in G_{\alpha},~g_{2} \in G_{\beta}$ such that $\varphi(g_{1}+Z(G))=h_{1}+Z(H)$ and $\varphi(g_{2}+Z(G))=h_{2}+Z(H)$. since $G\sim H$, we have $\theta[g_{1},g_{2}]=[h_{1},h_{2}]$. As a consequence, $\theta((ad_{g_{1}})^{n}(g_{2}))=(ad_{h_{1}})^{n}(h_{2})=0$. Hence $H$ is nilpotent with nilindex n.
\end{proof}
\begin{proposition}
 Let $H$ be a nilpotent Lie superalgebras with nilindex 2. Then $SDer_{z}(H)$ has a central subalgebra isomorphic to $T(H/Z(H),H^{1})$ containing $SIDer_{z}(H)$ as a super vector subspace.
\end{proposition}
\begin{proof}
By Lemma 2.7 and [7, Lemma 4.2], we have a stem Lie superalgebra $G$ isoclinic to $H$ with nilindex 2. Clearly $Z(G)=G^{1}\cong H^{1}$. Also, $H/Z(H)\cong G/Z(G)\cong G/G^{1}$. Thus $T(G/G^{1},Z(G))\cong T(H/Z(H),H^{1})$ and from corollary2.6 $N\cong T(H/Z(H),H^{1})$.\\
Also, for $h_{1} \in H_{\alpha}$, the map $T_{h_{1}}^{\prime}:H/Z(H)\rightarrow H^{1}$ defined by $T_{h_{1}}^{\prime}(h_{2}+Z(H))=ad_{h_{1}}(h_{2})$ for all $h_{2} \in H_{\beta}$ is a linear transformation. It is easy to see that the map $\sigma: H\rightarrow T(H/Z(H),H^{1})$ defined by $\sigma(h_{1})=T_{h_{1}}^{\prime}$ is a linear transformation with $ker(\sigma)=Z(H)$. Hence $H/Z(H)\cong \sigma(H)=SIDer(H)\subseteq T(H/Z(H),H^{1})$.
\end{proof}
 Let $H$ be a Lie superalgebra such that $H^{1}$ is abelian and denote $\kappa(H)$ to be the intersection of all $Ker(f)$, where $f$ ranges over all Lie homomorphisms from $H$ to $H^{1}$.
\begin{lemma}
Let $H$ be a Lie superalgebra with nilindex 2. Then $H^{1}=\kappa(H)$.
\end{lemma}

\begin{proof}
Let $h_{1},h_{2} \in H$ and $f:H\rightarrow H^{1}$ be a Lie homomorphism. Since $ H^{2}=[ H^{1},H]=0$, we have  $f([h_{1},h_{2}])=[f(h_{1}),f(h_{2})]=0$. Therefore $H^{1}\subseteq\kappa(H)$. Now suppose there exist  $h \in\kappa(H)$ such that $h \notin H^{1}$. Then we can write $H=L \oplus <h>$. Clearly $L$ is a maximal subalgebra of $H$ and hence by [6,Theorem 3.9], $L$ is an ideal. For a fixed $z \in H^{1}$,  define $f:H\rightarrow H^{1}$ by $f(l+\lambda h):=\lambda z$ for every $l \in L$ and $\lambda$ is a scalar. It is easy to show $f$ is a well-defined Lie superalgebra homomorphism with $ker(f)=L$. Hence we get a contradiction. Thus $H^{1}\supseteq \kappa(H)$ and $H^{1}=\kappa(H)$.

\end{proof}
\begin{theorem}
Let $H$ be a nilpotent Lie superalgebra with nilindex 2. Then $Z(SDer_{z}(H)) \cong T(H/Z(H),H^{1}).$
\end{theorem}
\begin{proof}
 From corollary 2.8, $N\cong T(H/Z(H),H^{1})$, where $N=\{ T^{*}|T \in SDer_{z}(G)\}$ is a central subalgebra. To complete the prove of the theorem, we will show $Z(SDer_{z}(H))\subseteq N$. For each Lie homomorphism $f:H\rightarrow H^{1}$, we define $\tau_{f}:H\rightarrow H$ by $\tau_{f}(h_{\alpha}):=f(h_{\alpha})$, where $h_{\alpha} \in H_{\alpha}$. Since $Z(H)=H^{1}$, so $\tau_{f} \in SDer_{z}(H)$. Let $T \in Z(SDer_{z}(H))$, then $[T,\tau_{f}]=0$. Therefore, $\tau_{f}(T(h_{\alpha}))=(-1)^{deg(T)deg(\tau_{f})}T(\tau_{f}(h_{\alpha}))$. Since $\tau_{f}(h_{\alpha}) \in H^{1}$ and $T$ is a central derivation, we have $\tau_{f}(T(h_{\alpha}))=0=f(T(h_{\alpha}))$. Hence, $T(h_{\alpha}) \in H^{1}=\kappa(H)$. Let $z_{\alpha} \in Z(H)$, $L$ be a maximal subalgebra of $H$ containing $Z(H)$ and $x_{\beta} \notin L$. Then $L$ is an ideal of $H$ and there is a map $\chi:H\rightarrow Z(H)$ such that $Z(H)\subseteq kernel(\chi)$ and $\chi(h_{\alpha})=z_{\beta}$. Now we define $\sigma_{\chi}:H\rightarrow H$ by $\sigma_{\chi}(h_{\alpha}):=\chi(h_{\alpha})$.
 Since $T \in Z(SDer_{z}(H))$, we have $T(z_{\beta})=T(\sigma_{\chi}(h_{\alpha}))=(-1)^{deg(T)deg(\sigma_{\chi})}\sigma_{\chi}(T(h_{\alpha}))=0$. Hence $T(Z(H))=0$. The map $T_{1}:H\rightarrow H$ defined by $T_{1}(g_{\alpha})=\theta^{-1}(T(h_{\alpha}))$, where $\varphi(g_{\alpha}+Z(G))=h_{\alpha}+Z(H)$, is well defined. Then $T_{1} \in SDer_{z}(G) $ and $T^{*}_{1}=T$.
\end{proof}

\begin{lemma}
Let $H$ be a nilpotent Lie superalgebra with nilindex 2. Then $SDer_{z}(H)$ is abelian if and only if $H^{1}=Z(H)$.
\end{lemma}
\begin{proof}
If $H^{1}=Z(H)$, then $H$ is a stem algebra and by proposition 2.1,  $SDer_{z}(H)$ is abelian Lie superalgebra. Now, suppose  $SDer_{z}(H)$ is abelian, then by proposition 2.5, $SDer_{z}(H)\cong T(H/H^{1},Z(H))$ and on the other hand , by theorem 3.10,   $SDer_{z}(H)=Z(SDer_{z}(H)) \cong T(H/Z(H),H^{1})$. Therefore, $ T(H/Z(H),H^{1})\cong T(H/H^{1},Z(H))$. Let $Z(H)=K\oplus H^{1}$ as a super vector spaces. Then
\[T(H/H^{1},Z(H))=T(H/H^{1},H^{1})\oplus T(H/H^{1},K).\]
Since $T(H/Z(H),H^{1})$ is isomorphic to a subalgebra of $T(H/H^{1},H^{1})$, we have $ T(H/H^{1},K)=0$. Therefore $K=0$ and $H^{1}=Z(H)$.
\end{proof}
{}
\end{document}